\documentclass[aps,pra,showpacs]{revtex4}
\usepackage{epsfig}
\usepackage{amsbsy,latexsym}
\usepackage{amsmath}
\usepackage{amssymb, mathrsfs}
\usepackage[mathscr]{eucal}
\usepackage{hyperref}
\usepackage{amsfonts}
\usepackage{amsthm}
\usepackage{amsmath}
\usepackage{amsfonts}
\usepackage{latexsym}
\usepackage{amssymb}
\usepackage{amscd}
\usepackage[latin1]{inputenc}
\usepackage{verbatim}

\usepackage[english]{babel}

\newtheorem{theorem}{Theorem}[section]
\newtheorem{proposition}[theorem]{Proposition}
\newtheorem{lemma}[theorem]{Lemma}
\newtheorem{corollary}[theorem]{Corollary}

\theoremstyle{definition}
\newtheorem{defn}[theorem]{Definition}

\newtheorem*{remark}{Remark}

\newcommand{\F}{\mathcal{F}}
\newcommand{\B}{\mathcal{B}}

\renewcommand{\H}{\mathcal{H}}
                       
\newcommand{\E}{\mathbb{E}}
\newcommand{\N}{\mathbb{N}}
\newcommand{\R}{\mathbb{R}}
\newcommand{\C}{\mathbb{C}}
\newcommand{\QCE}[3]{\E_{#1}\left[ {#2} | {#3} \right]} 
\newcommand{\QE}[2]{\E_{#1}\left[ {#2}\right]} 
\newcommand{\ch}[1]{\chi_{#1}}
 
\newcommand{\tr}{ \operatorname{Tr} } 
 
\newcommand{\state}[1]{\mathcal{S}(#1)}

\newcommand{\ac}{ \ll_{\rm ac}}
\newcommand{\dd}{\mathrm{d}}
\renewcommand{\d}{\, \mathrm{d}}
\newcommand{\borel}[1]{\mathcal{O}(#1)}
\newcommand{\povm}[2]{\operatorname{POVM}_{#2}({#1})}
\newcommand{\povpm}[2]{\operatorname{POVM}^{1}_{#2}({#1})}

\newcommand{\define}{\emph}
\newcommand{\geo}{\#}
\newcommand{\eps}{\varepsilon}

\newcommand\cstarconv{ {\rm C}^*{\rm conv}}

\begin{document}
\title{Conditional expectation and Bayes' rule for quantum random variables and
positive operator valued measures}

  \author{Douglas Farenick}
  \email{douglas.farenick@uregina.ca}
  \affiliation{Department of Mathematics and Statistics, University of Regina, Regina, Saskatchewan S4S 0A2, Canada}
 
\author{Michael J.~Kozdron}
\email{kozdron@stat.math.uregina.ca}
\affiliation{Department of Mathematics and Statistics, University of Regina, Regina, Saskatchewan S4S 0A2, Canada}

\date{ \today}

\begin{abstract}
A quantum probability measure $\nu$ is a function on a $\sigma$-algebra of subsets of a (locally compact and Hausdorff) sample space
that satisfies the formal requirements for a measure, but where the values of $\nu$ are positive operators acting on a complex Hilbert space, 
and a quantum random variable is a measurable operator valued function. Although quantum probability measures and random variables are
used extensively in quantum mechanics, some of the fundamental probabilistic features of these structures remain to be determined. In this
paper we take a step toward a better mathematical understanding of quantum random variables and quantum probability measures by
introducing a quantum analogue for the expected value $\mathbb E_\nu[\psi]$ of a quantum random variable $\psi$ relative to a quantum
probability measure $\nu$.  In so doing we are led to theorems for a change of quantum measure  and a change of quantum variables. 
We also introduce a quantum conditional expectation  which results in quantum versions of some standard identities for Radon-Nikod\'ym derivatives. 
This allows us to formulate and prove a quantum analogue of Bayes' rule. 
\end{abstract}


\maketitle
\section*{Introduction}

The probabilistic aspects of quantum theory have led mathematical and theoretical physicists to consider so-called quantum analogues
of commonly used notions in classical probability theory. By the term ``quantum'' one usually means ``operator valued.'' Thus, in such
language, a quantum probability measure is a function $\nu$ defined on a $\sigma$-algebra $\F(X)$ of subsets of a sample space $X$
such that $\nu$ satisfies the formal requirements of a measure, but where the values of $\nu$ are not nonnegative real numbers, but rather 
they are quantum effects---namely, selfadjoint operators acting on a complex Hilbert space such that, for every $E\in\F(X)$,
the spectrum of the operator $\nu(E)$ is contained in the closed unit interval of $\R$. 

The principal benefit of passing to quantum analogues of classical mathematics and classical physics 
is to be found in the fact that the inherent structure of some object under study may have
properties that are revealed only through the use of quantum methods and are not observed at all through classical methods. There are many such
examples of this approach, such as the theory of operator spaces, which has clarified and enriched Banach space theory, and noncommutative differential 
geometry, which has brought new tools to bear upon our understanding the physical world.

Despite moving from real or complex numbers to selfadjoint or arbitrary Hilbert space operators, one nevertheless wishes the quantum analogue to simultaneously
capture the essence of the classical world and recover the classical world when, in this later situation, the Hilbert space is assumed to be one-dimensional.
In this regard, to be truly meaningful any quantum analogue of a classical theory must overcome two unavoidable features: 
(i) the noncommutativity of operator algebra, and (ii)
the (partial) order structure in the real vector space of selfadjoint operators. To illustrate this point, suppose that $a$ and $b$ are two positive operators
acting on a Hilbert space. If $\H$ has dimension at least $2$, then it is possible that $a$ and $b$ do not commute and in this case neither $ab$ nor $ba$
will be a positive operator, unlike the corresponding situation for real numbers or real-valued
functions; however, $b^{1/2}ab^{1/2}$ and $a^{1/2}ba^{1/2}$ are both positive operators regardless of whether $a$ and $b$ commute. 
(Here $h^{1/2}$ is used to denote the unique positive
square root of a positive operator $h$.)
Indeed, this process of symmetrisation to preserve positivity
will be a recurring technique in our work herein.

In this paper we introduce a quantum analogue of the expected value of a quantum random variable using an operator valued integral introduced and studied
in~\cite{FPS} and~\cite{FZ}, and which has the properties one desires of an expected value, such as linearity, monotonicity, and so forth. 
Furthermore, by introducing a multiplication $\boxtimes$ for the product of a quantum random variable with a quantum Radon-Nikod\'ym derivative, we formulate and prove a quantum analogue of the change of measure theorem. With this result we are then led to establish quantum analogues of some of the fundamental features of the Radon-Nikod\'ym derivative such as quantum conditional expectation and the chain rule. We also establish a formula for a change of quantum variables, a quantum conditional Jensen's inequality, and 
a quantum version of Bayes' rule.

The theory of positive operator valued measures appears in quantum theory via the measurement postulate.
While conditional probabilities for quantum measurements have been considered in~\cite{kl2004},~\cite{SBC} and, more recently, in~\cite{LS2011a}, 
conditional expectations are more problematic. Indeed, quantum formulations of
conditional expectation and Bayes' rule have yet to be settled from the epistemological perspective. 
Conditional expectations in operator algebras, as in Section~9.2 of~\cite{petz}, are fairly natural, but do not necessarily directly
capture probabilistic notions of immediate interest to physicists. 
In this context, there are two questions to 
confront. (i) Does a given formulation of Bayes' rule adhere to the principles of quantum theory? (ii) What is the 
physical or ontological meaning of any mathematically valid formulation of Bayesian inference in quantum theory?
Concerning the second question, the recent literature reveals a substantial body of discussion; we mention here the article~\cite{fuchs--schack2004} as
an entry point into the debate. Concerning the first point, one must keep in mind that obtaining information about a system in a given state
generally alters the state of the system, and so a quantum Bayes' rule must take into account both the updating arising from information about an event and
the disturbing effects of measurement. In this regard, if $\psi(x)\equiv\rho$ is a constant state-valued
quantum random variable on a quantum probability space $(X,\borel X, \nu)$
and if $\mathcal F(X)$ is a sub-$\sigma$-algebra of known events, then we view the quantum conditional expectation $\QCE{\nu}{\psi}{\F(X)}$, as defined in~\S\ref{condexpsect},
as a ``state of belief'' rather than a ``state of nature,''
in keeping with the interpretation of the notion of conditional expectation put forward in~\cite{fuchs--schack2004}.

\section{Quantum conditional expectation}\label{Quantum conditional expectation}

\subsection{Motivating concept: quantum averaging}

A (classical) convex combination $\sum_{j=1}^nt_j\psi(x_j)$
of the values of a function $\psi$ defined on a finite set $X=\{x_1,\dots,x_n\}$ can be viewed as the expected value of $\psi$ relative to the probability
distribution corresponding to the convex coefficients $t_1,\dots, t_n\in[0,1]$. That is, 
\begin{equation}\label{classical avg}
\int_X\psi\,\dd\mu\,=\,\sum_{j=1}^nt_j\psi(x_j)\,,
\end{equation}
where $\mu$ is the unique probability measure on the power set of $X$
for which $\mu(\{x_j\})=t_j$ for each $j$. Observe that here $t_1,\dots,t_n$ are nonnegative real numbers that sum to $1$, but there need not be
any restriction whatsoever on where the values of $\psi$ lie other than that $\psi(x_1),\dots,\psi(x_n)$ be contained in some vector space.

Suppose now that $\H$ is a complex Hilbert space and that $\B(\H)$ is the C$^*$-algebra of bounded linear operators on $\H$. One would like to
consider a quantised formulation of~\eqref{classical avg} so that if the values of $\psi$ lie in $\B(\H)$, then the convex coefficients $t_1,\dots,t_n$ do so
as well. However, because $\B(\H)$ is noncommutative, the sum in~\eqref{classical avg} will not in general result in a positive operator, even if each $t_j$ and $\psi(x_j)$
are positive operators. To ensure preservation of positivity, the summation in~\eqref{classical avg} must be symmetrised---see~\eqref{quantum avg} below.

To this end, let $\B(\H)_+$ denote the cone of positive operators and suppose that the sum of  $h_1,\dots,h_n\in\B(\H)_+$ is the identity operator $1\in\B(\H)$. 
If $\psi:X\rightarrow\B(\H)$ is a function, then the operator
\begin{equation}\label{quantum avg}
\sum_{j=1}^nh_j^{1/2}\psi(x_j)h^{1/2}
\end{equation}
is the \emph{quantum expected value} of $\psi$ is relative to the operator valued probability measure $\nu$ for which $\nu(\{x_j\})=h_j$ for each $j$. The quantum
average~\eqref{quantum avg} preserves positivity; furthermore, it is suggestive and natural to use the notation
\begin{equation}\label{QE discrete}
\QE{\nu}{\psi} = \int_X\psi\,\dd\nu\,=\,\sum_{j=1}^nh_j^{1/2}\psi(x_j)h^{1/2}
\end{equation}
to denote the quantum expected value of $\psi$. 

Although~\eqref{quantum avg} is a generalised convex combination of operators $\psi(x_1),\dots,\psi(x_n)$, this represents a special form of a more general notion of
nonclassical convexity. A 
\emph{C$^*$-convex combination} of operators $z_1,\dots,z_m\in\B(\H)$ is an operator $z$ of the form
\begin{equation}\label{cstarconv comb}
z\,=\,\sum_{j=1}^m a_j^*z_j a_j\,,\;\mbox{ where } a_1,\dots,a_m\in\B(\H)\;\mbox{ are such that }
\sum_{j=1}^m a_j^*a_j\,=\,1\,.
\end{equation}
If $\Gamma\subset\B(\H)$ is a nonempty subset, then the \define{C$^*$-convex hull} of $\Gamma$ is the set
$\cstarconv\,\Gamma$ consisting of all operators $z$ of the form~\eqref{cstarconv comb}, where $m\in\mathbb N$ is
arbitrary and $z_1,\dots,z_m\in\Gamma$. The C$^*$-convex hull is, in general, not a closed set; however,  
it is known~\cite{farenick1992} that  if $\Gamma$ is
compact and the dimension $d$ of $\H$ is finite, then $\cstarconv\,\Gamma$ is compact. In particular, the C$^*$-convex hull of a single operator $z\in\B(\H)$, namely
\begin{equation}\label{cstarconv compact}
\cstarconv\,\{z\}\,=\,\left\{ \sum_{j=1}^m a_j^*z a_j\,:\,m\in\mathbb N,\;\sum_{j=1}^m a_j^*a_j\,=\,1\right\}\,,
\end{equation}
is compact and, as shown in~\cite{farenick1992}, the number $m$ of summands $a_j^*za_j$ required to exhaust $\cstarconv\,\{z\}$ is uniformly bounded above by a
polynomial in the (finite) dimension $d$ of the Hilbert space $\H$.

\subsection{Measure and integration}

Henceforth $X$ shall denote a locally compact Hausdorff space
and $\F(X)$ a $\sigma$-algebra of subsets of $X$. A particular $\sigma$-algebra of interest is the $\sigma$-algebra
of Borel sets of $X$, which is denoted by $\borel{X}$. 

Assume that $\H$ is a separable Hilbert space with canonical trace functional $\tr(\cdot)$.
A \emph{density operator}, or \emph{state}, on $\H$ is a positive trace-class operator $\rho$ such that $\tr(\rho)=1$; the set
of all density operators is denoted by $\state{\H}$.

A function $\nu: \F(X) \to \B(\H)$ is called a \define{positive operator valued measure} on $(X,\F(X))$ if
\begin{enumerate}
\item $\nu(E) \in\B(\H)_+$ for every $E \in \F(X)$,
\item $\nu(X) \neq 0$, and
\item for every countable collection $\{E_k\}_{k \in \N} \subseteq \F(X)$ with $E_j \cap E_k = \emptyset$ for $j \neq k$ we have
$$\nu\left(\bigcup_{k\in \N} E_k \right) = \sum_{k \in \N} \nu(E_k)$$
where the convergence on the right side of the equation above is with respect to the $\sigma$-weak topology of $\B(\H)$.
\end{enumerate}

We will write $\povm{X, \F(X)}{\H}$ for the set of all positive operator valued measures on $(X,\F(X))$ with values in $\B(\H)$. 
In the case that $\F(X)=\borel{X}$, we will drop the $\borel{X}$ from the notation and write $\povm{X}{\H}$. 
By $\povpm{X, \F(X)}{\H}$ we mean those $\nu\in \povm{X, \F(X)}{\H}$ satisfying $\nu(X)=1$ and similarly for $\povpm{X}{\H}$. Such a $\nu$ will often be called either a 
\emph{positive operator valued probability measure} or a \emph{quantum probability measure}.

A function $\psi: X \to \B(\H)$ is said to be \define{$\F(X)$-measurable} if, 
for every state (that is, density operator) $\rho \in\B(\H)$, the complex-valued function $\omega_\rho  :X \to \C$ given by
\[
\omega_\rho(x)\,=\,   \tr(\rho \psi(x))
\]
is $\F(X)$-measurable (in the sense that $\omega_\rho^{-1}(U) \in \F(X)$ for every open set $U \subset \C$).

Each $\nu \in \povm{X, \F(X)}{\H}$ gives rise to
a measure $\mu=\mu_\nu$ on $(X, \F(X))$ via
\begin{equation}\label{induced msr}
\mu (E)\,=\,\tr\left(\nu(E)\right),\;E\in\F(X)\,.
\end{equation}
If $\H$ has finite dimension $d$, then the measure $\mu $ above is assumed to be normalised to
\begin{equation}\label{povmprobmeas}
\mu\,=\,\frac{1}{d}\,\tr\circ\nu 
\end{equation}
so that $\mu$ is a probability measure if $\nu$ is.
As explained in~\cite{FPS}, the
Radon-Nikod\'ym derivative $\dd\nu/\dd\mu$ is a quantum random variable
$X\rightarrow\B(\H)_+$ and admits a measurable positive square root $\left( \dd\nu/\dd\mu\right)^{1/2}$.

\begin{defn}\label{defnofnuintegrable}
If $\psi: X \to \B(\H)$ is a quantum random variable and if $\nu \in \povm{X}{\H}$, 
then $\psi$ is said to be \define{$\nu$-integrable} if for every state $\rho$  the complex-valued function $\psi_\rho :X \to \C$ defined by
\begin{equation}\label{muint}
\psi_\rho(x) = \tr \left(
 \rho \left(\frac{\dd \nu}{\dd \mu}(x) \right)^{1/2} \psi(x) \left(\frac{\dd \nu}{\dd \mu}(x) \right)^{1/2} 
\right)
\end{equation}
is $\mu$-integrable. In this case, the \define{integral of $\psi$ with respect to $\nu$} is the unique operator $\int_X \psi \d \nu$ on $\H$ with 
the property that
$$\tr \left( \rho\int_X \psi \d \nu \right) = \int_X \psi_\rho \d \mu$$
for every state $\rho$.
\end{defn}

\begin{defn}
If $\nu \in \povpm{X}{\H}$, then the
 \define{quantum expectation} (or \define{quantum expected value})  
of $\psi$ relative to the quantum probability measure $\nu$
is the operator denoted by $\QE{\nu}{\psi}$ and defined by
\[
\QE{\nu}{\psi} \,=\,  \int_X \psi\, \d \nu\,.
\]
\end{defn}

\subsection{Properties of quantum expectation}

We focus now on a finite-dimensional Hilbert space $\H$, a compact Hausdorff space $X$, and the $\sigma$-algebra $\borel{X}$ of Borel sets, 
although several aspects of the following theorem remain true for arbitrary 
Hilbert spaces $\H$, locally compact $X$, and arbitrary $\sigma$-algebras $\F(X)$. 

In what follows below, $\chi_E$ denotes the indicator (that is, characteristic) function of $E\in\borel{X}$, $C(X)$ is the abelian C$^*$-algebra of all complex-valued continuous
functions on a compact Hausdorff space $X$, and a unital completely positive linear map $\mathcal E:\B(\H)\rightarrow\B(\H)$ is a  conditional expectation 
if $\mathcal E\circ\mathcal E=\mathcal E$.

\begin{theorem}\label{properties of QE} If
 $\H$ is a Hilbert space of finite dimension $d$, $X$ is compact, $\nu \in \povpm{X}{\H}$, and $\psi,\psi_1,\psi_2:X\rightarrow\B(\H)$
are quantum random variables, then the following assertions hold:
\begin{enumerate}
\item\label{qe1} (Generalised linearity) $\QE{\nu}{\varrho_1\psi_1 + \varrho_2\psi_2} = \varrho_1\QE{\nu}{\psi_1} + \varrho_2\QE{\nu}{\psi_2}$
for all  $\varrho_1$, $\varrho_2 \in \B(\H)$ that commute with the range of 
$\dd \nu/\dd \mu$;
\item\label{ge2} (Monotonicty) $\QE{\nu}{\psi}\in\B(\H)_+$ if $\psi(x)\in\B(\H)_+$ for $\mu$-almost all $x\in X$;
\item\label{qe3} (Additivity) if $E, F \in \borel{X}$ are such that $E \cap F = \emptyset$, then $\QE{\nu}{\chi_{E\cup F}\psi}=\QE{\nu}{\chi_{E}\psi}+\QE{\nu}{\chi_{ F}\psi}$;
\item\label{qe4} (Finitely supported measures) if $\nu$ is supported on a finite set $ \{x_1, \ldots, x_n\}\subset X$, then
$$\QE{\nu}{\psi} = \sum_{j=1}^n h_j^{1/2}\psi(x_j)  h_j^{1/2},$$
where $h_j=\nu(\{x_j\})$ for $j=1,\dots,n$;
\item\label{qe5} (Complete positivity) the map
$\phi_\nu:C(X)\otimes\B(\H)\rightarrow\B(\H)$ defined by
\[
\phi_\nu(f)\,=\,\int_X f\,d\nu
\]
is a unital completely positive linear map;
\item\label{qe6} {\rm (Jensen's inequality)}
if $J\subset\mathbb R$ is an open interval containing a closed interval
$[\alpha,\beta]$, and if
$\psi(x)$ is selfadjoint and has spectrum contained in
$[\alpha,\beta]$
for every $x\in X$, then 
\[
\vartheta\left(\int_X \psi\,d\nu\right)\,\leq\,\int_X\vartheta\circ \psi\,d\nu\,,
\]
for every operator convex function $\vartheta:J\rightarrow\mathbb R$;
\item\label{qe7} (Quantum expectation of constant functions) the linear map
$\mathcal E_\nu:\B(\H)\rightarrow\B(\H)$ defined by
\[
\mathcal E_\nu(z) = \int_X z \d \nu \, , \;\; z\in\B(\H),
\]
is a unital quantum channel---hence, $\QE{\nu}{\rho}$ is a state for every state $\rho$;
\item\label{qe8} (Fixed points) the set $\mathcal A_\nu=\{z\in\B(\H)\,:\,\QE{\nu}{z}=z\}$ is a unital C$^*$-subalgebra of $\B(\H)$;
\item\label{qe9} (Ergodic Property) there exists a trace-preserving conditional expectation $\mathfrak E_\nu:\B(\H)\rightarrow\B(\H)$ with range $\mathcal A_\nu$ such that
\[
\lim_{N\rightarrow\infty}\,\frac{1}{N}\left(\mathcal I+\sum_{j=1}^{N-1}\underbrace{\mathcal E_v \circ \cdots \circ \mathcal E_\nu}_{j}\right)\,=\,\mathfrak E_\nu\,.
\]
\end{enumerate}
\end{theorem}

\begin{proof} Statements~\eqref{qe1} to~\eqref{qe4} follow readily from the definitions (and do not require $\H$ to be of finite dimension or $X$ to be compact). Statement~\eqref{qe5} is 
established in~\cite{FPS} and statement~\eqref{qe6} is the main result of~\cite{FZ}.

For the proof of~\eqref{qe7}, by Corollary~3.4  of~\cite{FPS} there is a
net $\{\nu_\alpha\}_\alpha\subset \povpm{X}{\H}$ such that each $\nu_\alpha$ has finite support and
$$\int_X f\d\nu_\alpha\rightarrow\int_X f\d\nu$$
 for every continuous function $f:X\rightarrow\B(\H)$.
Hence, by assertion~\eqref{qe4}, for each $\alpha$
there exist $m_\alpha\in\mathbb N$, distinct points $x_{1,\alpha}, \dots,x_{m_\alpha,\alpha}\in X$, and quantum effects $h_{1,\alpha},\dots,h_{m_\alpha,\alpha}$
such that 
\[
\int_X f\d\nu_\alpha = \sum_{j=1}^{m_\alpha}   h_{j,\alpha}^{1/2}f(x_{j,\alpha})  h_{j,\alpha}^{1/2}  
\]
for every $f\in C(X)\otimes\B(\H)$.

Now fix $z\in\B(\H)$ and consider the constant function $f(x)\equiv z$, which is trivially continuous. Thus, 
\[
\int_X z \d\nu = \lim_\alpha \sum_{j=1}^{m_\alpha}   h_{j,\alpha}^{1/2}z  h_{j,\alpha}^{1/2}  
\in
\overline{\left(\cstarconv\,\{z\}\right)}^{\|\cdot\|}\,,
\]
where $\overline{\Omega}^{\|\cdot\|}$ denotes the closure in the norm-topology of a subset $\Omega\subset\B(\H)$. As mentioned
earlier, the set $\cstarconv\,\{z\}$ is already compact. Thus,
\[
\mathcal E_\nu(z) 
\in
\cstarconv\,\{z\}\,.
\]
Using the fact that the trace functional is continuous, we also deduce from the same approximation that
\[
\tr\left(\mathcal E_\nu(z)\right) = \tr \left(\lim_\alpha\sum_{j=1}^{m_\alpha} h_{j,\alpha}^{1/2} z  h_{j,\alpha}^{1/2}  \right) 
=\lim_\alpha \sum_{j=1}^{m_\alpha} \tr \left( h_{j,\alpha}^{1/2} z  h_{j,\alpha}^{1/2}  \right)
= \lim_\alpha\tr \left( z   \sum_{j=1}^{m_\alpha} h_{j,\alpha} \right)\\
= \tr(z)\,,
\]
which proves that $\mathcal E_\nu$ is trace preserving. The function $\mathcal E_\nu$ is also unit preserving since $\nu(X)=1$.
Therefore, what remains is to verify that $\mathcal E_\nu$ is completely positive. 

To this end, select $k\in\N$ and consider $M_k\left( \B(\H)\right)$. 
We are to prove that if $[z_{ij}]_{i,j=1}^k$ is positive in $M_k\left( \B(\H)\right)$, 
then  $[\mathcal E_\nu(z_{ij})]_{i,j=1}^k$ is positive as well. For each $\alpha$, the linear map $\mathcal E_{\nu_\alpha}$ is given by
\[
\mathcal E_{\nu_\alpha}(z)\,=\,\sum_{j=1}^{m_\alpha} h_{j,\alpha}^{1/2} z  h_{j,\alpha}^{1/2} ,\;\,z\in\B(\H)\,,
\mbox{ where }
\sum_{j=1}^{m_\alpha}h_{j,\alpha}\,=\,1\,.
\]
Thus, $\mathcal E_{\nu_\alpha}$ has the structure of a unital quantum channel.
Therefore, each 
$\mathcal E_{\nu_\alpha}$ is completely positive, and so $[\mathcal E_{\nu_\alpha}(z_{ij})]_{i,j=1}^k$is positive and
the equation
\[
\left[\mathcal E_\nu(z_{ij})\right]_{i,j}\,=\,\left[\lim_\alpha \mathcal E_{\nu_\alpha}(z_{ij})\right]_{i,j}
\,=\,
\lim_\alpha \left([\mathcal E_{\nu_\alpha}(z_{ij})]_{i,j}\right)
\]
expresses $\left[\mathcal E_\nu(z_{ij})\right]_{i,j}$ as a limit of positive operators. Hence, 
$\left[\mathcal E_\nu(z_{ij})\right]_{i,j}$ is positive, which implies that $\mathcal E_\nu$ is completely positive.

To prove~\eqref{qe8}, it is well known~\cite{kribs2003},~\cite{lindblad1999} that the fixed points of a unital quantum channel form a C$^*$-algebra. Therefore, 
the fixed point space $\mathcal A_\nu$ of $\mathcal E_\nu$ is a unital C$^*$-algebra.

Lastly, the ergodic property~\eqref{qe9} is an immediate consequence of Corollary~5.3 in~\cite{farenick2011b} applied to the unital quantum channel $\mathcal E_\nu$.
\end{proof}

\section{Calculus}

\subsection{Radon-Nikod\'ym theorem}

\begin{defn} If $\nu_1,\nu_2\in\povm{X,\F(X)}{\H}$,  then $\nu_2$ is \define{absolutely continuous} with respect to $\nu_1$, which we denote by $\nu_2 \ac \nu_1$, if 
$\nu_2(E) = 0$ for every $E \in \F(X)$ with $\nu_1(E) = 0$.  
\end{defn}

For the remainder of the present paper, the notation $a^{-1}$ for positive operator $a\in\B(\H)$ refers to a \define{generalised inverse}. That is,
by way of the spectral theorem, $a^{-1}$ for $a\in\B(\H)_+$ is defined by
\begin{equation}\label{invdef}
a^{-1}=\sum_{\lambda_j\neq0} \frac{1}{\lambda_j}\,p_j\,,
\end{equation}
where $p_1,\ldots,p_m\in\B(\H)$ are (pairwise-othogonal) projections
such that $p_1+\cdots+p_m=1$,  $\lambda_1,\ldots,\lambda_m$ are the distinct eigenvalues of $a$, and $a=\lambda_1p_1+\cdots+\lambda_mp_m$ is the 
spectral decomposition of $a$.
In the case where no eigenvalue of $a$ is zero, the definition of $a^{-1}$ given in~\eqref{invdef} above recovers the inverse of $a$ in the usual sense. However, in general, 
$a^{-1}a=aa^{-1}=q$, where $q\in\B(\H)$ is the (unique) projection onto the range of $a$ with kernel satisfying $\ker q=\ker a$.
The following Radon-Nikod\'ym theorem is given in Theorem~2.7 of~\cite{FPS}.

\begin{theorem}\label{NPRNDthm}
If  $\nu_1,\nu_2\in\povm{X, \F(X)}{\H}$ and $\H$ is of finite dimension, then the following statements are equivalent.
\begin{enumerate}
\item $\nu_2 \ll_{\rm ac} \nu_1$.
\item There exists a bounded $\nu_1$-integrable $\F(X)$-measurable function $\varphi:(X, \F(X))\to \B(\H)$, unique up to sets of $\nu_1$-measure zero, such that
\begin{equation}\label{rn prop}
\nu_2(E) = \int_E \varphi \d \nu_1
\end{equation}
for every $E\in \F(X)$.
\end{enumerate}
Moreover, if the equivalent conditions above hold and if $\mu_j= \mu_{\nu_j}$ is the finite Borel measure induced by $\nu_j$
as given in~\eqref{povmprobmeas}, then
$\mu_2\ac \mu_1$ and
\begin{equation}\label{rn-deriv}
\varphi= \left(\frac{\dd\mu_2}{\dd\mu_1}\right)
\left[ 
\left(\frac{\dd\nu_1}{\dd\mu_1}\right)^{-1/2}\left(\frac{\dd\nu_2}{\dd\mu_2}\right)\left(\frac{\dd\nu_1}{\dd\mu_1}\right)^{-1/2}
\right]\,.
\end{equation}
\end{theorem}

\begin{defn}
The quantum random variable $\varphi$ arising in Theorem~\ref{NPRNDthm}
is called the \define{Radon-Nikod\'ym derivative} of $\nu_2$ with respect to $\nu_1$ and is denoted by
$$\frac{\dd\nu_2}{\dd\nu_1}=\varphi.$$
\end{defn}
 
\subsection{Change of quantum measure}\label{changesect}
 
The geometric mean of two positive operators was first introduced in~\cite{Pusz} and its basic properties were further examined in~\cite{KuboAndo}. 
If $a,b \in \B(\H)_+$ are both invertible, then the \define{geometric mean of $a$ and $b$} is the positive operator $a \geo b$ defined by  
\begin{equation}\label{gm1}
a \geo b = a^{1/2} (a^{-1/2} ba^{-1/2})^{1/2} a^{1/2}.
\end{equation}
Note that 
$$(a^{-1/2} ba^{-1/2})^{1/2}  = a^{-1/2}(a \geo b)a^{-1/2}$$
which implies
\begin{equation}\label{geo2}
(a^{1/2} ba^{1/2})^{1/2}  = a^{1/2}(a^{-1} \geo b)a^{1/2}.
\end{equation}
If $a$, $b \in \B(\H)_+$ are non-invertible, then $a \geo b$ is defined by
\begin{equation}\label{gm2}
a \geo b= \lim_{\eps\to0+} (a +\eps 1) \geo (b + \eps 1)\,,
\end{equation}
with convergence in the strong operator topology.

\begin{defn}\label{boxtimes}
Suppose that $\nu_1$, $\nu_2 \in \povpm{X}{\H}$ with  $\nu_2 \ac \nu_1$ and that  $\mu_j=\mu_{\nu_j}$  
is the induced Borel probability measures, as given in~\eqref{povmprobmeas}, for each $j=1,2$.  If $\psi: X \to \B(\H)$ is a 
quantum random variable, then
\begin{equation}\label{boxtimesdefn}
\psi \boxtimes \frac{\dd \nu_2}{\dd\nu_1}=
 \left(\left(\frac{\dd\nu_1}{\dd\mu_1}\right)^{-1}\geo\frac{\dd \nu_2}{\dd \nu_1} \right)\left(\frac{\dd\nu_1}{\dd\mu_1}\right)^{1/2} \psi \left(\frac{\dd\nu_1}{\dd\mu_1}\right)^{1/2} \left(\left(\frac{\dd\nu_1}{\dd\mu_1}\right)^{-1}\geo\frac{\dd \nu_2}{\dd \nu_1} \right).
 \end{equation}
\end{defn}

\begin{remark} 
In the commutative setting---and, in particular, in the classical case of $\H=\mathbb C$---the multiplication 
defined by~\eqref{boxtimesdefn} reduces to ordinary multiplication. That is, 
if $a,b \in \B(\H)_+$ commute, then $a\geo b= a^{1/2} b^{1/2} = b^{1/2} a^{1/2} = b \geo a$. Thus, if 
$\psi$, $\dd\nu_1/\dd\mu_1$, and $\dd\nu_2/\dd\nu_1$ are pairwise commuting, then
$$\psi \boxtimes \frac{\dd \nu_2}{\dd\nu_1} = \psi \frac{\dd \nu_2}{\dd\nu_1} =  \frac{\dd \nu_2}{\dd\nu_1}\psi.$$
\end{remark}

We will now state and prove one of our primary results, the change of quantum measure theorem.

\begin{theorem}[Change of Quantum Measure]\label{quantumchange}
Assume that $\H$ has finite dimension and that $\nu_1$, $\nu_2 \in \povpm{X}{\H}$ satisfy $\nu_2 \ac \nu_1$. If $\psi: X \to \B(\H)$ is a  $\nu_2$-integrable
quantum random variable, then
$$\psi \boxtimes \frac{\dd \nu_2}{\dd\nu_1}$$
as defined in~\eqref{boxtimesdefn}
is a $\nu_1$-integrable quantum random variable and
$$\QE{\nu_2}{\psi} = \QE{\nu_1}{\psi \boxtimes \frac{\dd \nu_2}{\dd\nu_1}}.$$
\end{theorem}

\begin{proof}
Assume that $\nu_2 \ac \nu_1$ and for $j=1,2$  let  $\mu_j=\mu_{\nu_j} $  be the induced Borel probability measures so that
$\nu_2 \ac \mu_2 \ac \nu_1 \ac \mu_1$.
We know from~\eqref{rn-deriv} that 
\begin{equation*}
\frac{\dd \nu_2}{\dd \nu_1} = \frac{\dd\mu_2}{\dd\mu_1}
\left[ 
\left(\frac{\dd\nu_1}{\dd\mu_1}\right)^{-1/2} \frac{\dd\nu_2}{\dd\mu_2} \left(\frac{\dd\nu_1}{\dd\mu_1}\right)^{-1/2}
\right]\,,
\end{equation*}
and so
\begin{equation}\label{rnderiv2}
\left(\frac{\dd\nu_1}{\dd\mu_1}\right)^{1/2} \frac{\dd \nu_2}{\dd \nu_1} \left(\frac{\dd\nu_1}{\dd\mu_1}\right)^{1/2}
= \frac{\dd\mu_2}{\dd\mu_1}
\frac{\dd\nu_2}{\dd\mu_2}.
\end{equation}
Let $\psi : X \to \B(\H)$ be a $\nu_2$-integrable quantum random variable  and consider $\QE{\nu_2}{\psi}$. By definition of the
quantum integral, 
$$\tr(\rho \QE{\nu_2}{\psi} ) 
=\int_X \tr \left(
 \rho \left(\frac{\dd \nu_2}{\dd \mu_2} \right)^{1/2} \psi \left(\frac{\dd \nu_2}{\dd \mu_2}\right)^{1/2} 
\right) \d \mu_2
$$
for every state $\rho \in \B(\H)$.
 However, by the classical change of measure theorem, we can write
 \begin{align*}
 \int_X \tr \left(
 \rho \left(\frac{\dd \nu_2}{\dd \mu_2} \right)^{1/2} \psi \left(\frac{\dd \nu_2}{\dd \mu_2}\right)^{1/2} 
\right) \d \mu_2
&=
\int_X \tr \left(
 \rho \left(\frac{\dd \nu_2}{\dd \mu_2} \right)^{1/2} \psi \left(\frac{\dd \nu_2}{\dd \mu_2}\right)^{1/2} 
\right)  \frac{\dd\mu_2}{\dd\mu_1}\d \mu_1\\
&=
\int_X \tr \left(
 \rho \left(\frac{\dd\mu_2}{\dd\mu_1}\frac{\dd \nu_2}{\dd \mu_2} \right)^{1/2} \psi \left(\frac{\dd\mu_2}{\dd\mu_1}\frac{\dd \nu_2}{\dd \mu_2}\right)^{1/2} 
\right)\d \mu_1.
\end{align*} 
Using~\eqref{rnderiv2}, we find
\begin{align*}
\int_X &\tr \left(
 \rho \left(\frac{\dd\mu_2}{\dd\mu_1}\frac{\dd \nu_2}{\dd \mu_2} \right)^{1/2} \psi \left(\frac{\dd\mu_2}{\dd\mu_1}\frac{\dd \nu_2}{\dd \mu_2}\right)^{1/2} 
\right)\d \mu_1\\
&=
\int_X \tr \left(
 \rho \left(\left(\frac{\dd\nu_1}{\dd\mu_1}\right)^{1/2} \frac{\dd \nu_2}{\dd \nu_1} \left(\frac{\dd\nu_1}{\dd\mu_1}\right)^{1/2}\right)^{1/2} \psi \left(\left(\frac{\dd\nu_1}{\dd\mu_1}\right)^{1/2} \frac{\dd \nu_2}{\dd \nu_1} \left(\frac{\dd\nu_1}{\dd\mu_1}\right)^{1/2}\right)^{1/2} 
\right)\d \mu_1
\end{align*}
and so as a consequence of~\eqref{geo2},
\begin{align*}
&\int_X \tr \left(
 \rho \left(\left(\frac{\dd\nu_1}{\dd\mu_1}\right)^{1/2} \frac{\dd \nu_2}{\dd \nu_1} \left(\frac{\dd\nu_1}{\dd\mu_1}\right)^{1/2}\right)^{1/2} \psi \left(\left(\frac{\dd\nu_1}{\dd\mu_1}\right)^{1/2} \frac{\dd \nu_2}{\dd \nu_1} \left(\frac{\dd\nu_1}{\dd\mu_1}\right)^{1/2}\right)^{1/2} 
\right)\d \mu_1\\
&=\!\int_X \!\tr \left(\!\rho \!\left(\frac{\dd\nu_1}{\dd\mu_1}\right)^{1/2} \!\left(\!\left(\frac{\dd\nu_1}{\dd\mu_1}\right)^{-1}\!\!\geo\frac{\dd \nu_2}{\dd \nu_1} \right)\!\left(\frac{\dd\nu_1}{\dd\mu_1}\right)^{1/2} \!\psi \left(\frac{\dd\nu_1}{\dd\mu_1}\right)^{1/2}\! \left(\left(\frac{\dd\nu_1}{\dd\mu_1}\right)^{-1}\!\!\geo\frac{\dd \nu_2}{\dd \nu_1} \right)\!\left(\frac{\dd\nu_1}{\dd\mu_1}\right)^{1/2}
   \right)\!\d \mu_1.
 \end{align*}
Hence, by Definition~\ref{boxtimes}, since
$$\psi \boxtimes \frac{\dd \nu_2}{\dd\nu_1} =  \left(\left(\frac{\dd\nu_1}{\dd\mu_1}\right)^{-1}\geo\frac{\dd \nu_2}{\dd \nu_1} \right)\left(\frac{\dd\nu_1}{\dd\mu_1}\right)^{1/2} \psi \left(\frac{\dd\nu_1}{\dd\mu_1}\right)^{1/2} \left(\left(\frac{\dd\nu_1}{\dd\mu_1}\right)^{-1}\geo\frac{\dd \nu_2}{\dd \nu_1} \right),$$
 substituting back into the previous expression and following the chain of equalities we deduce that
 $$ \int_X \tr \left(
 \rho \left(\frac{\dd \nu_2}{\dd \mu_2} \right)^{1/2} \psi \left(\frac{\dd \nu_2}{\dd \mu_2}\right)^{1/2} 
\right) \d \mu_2
=\int_X \tr \left(\rho \left(\frac{\dd\nu_1}{\dd\mu_1}\right)^{1/2}\psi \boxtimes \frac{\dd \nu_2}{\dd\nu_1} \left(\frac{\dd\nu_1}{\dd\mu_1}\right)^{1/2}
   \right)\d \mu_1.
$$
However, this is exactly the statement that
$$\psi\boxtimes \frac{\dd \nu_2}{\dd\nu_1}$$
is $\nu_1$-integrable and that
$$ \QE{\nu_2}{\psi}
=  \QE{\nu_1}{\psi\boxtimes \frac{\dd \nu_2}{\dd\nu_1}}\,,$$
which completes the proof.
 \end{proof}


\subsection{Change of quantum variables}\label{cofvsect}

Since the law of a random variable is a probability measure on the state space, this is, of course, what one needs when considering integrating with respect to the law.  Classically, the expectation of a random variable is rigorously defined as a Lebesgue integral with respect to a probability measure on the sample space. However, instead of computing this integral directly, the change of variables  formula allows one to compute an integral with respect to a probability measure on the state space. If the state space is $\R$, then the distribution function of the random variable characterizes its law and so the Lebesgue integral with respect to the law is equal to a Riemann-Stieltjes integral with respect to the distribution function which in turn reduces to an ordinary Riemann integral with respect to the density function of the random variable provided the law is absolutely continuous with respect to Lebesgue measure. As we will explain shortly, in the quantum case we have an analogous result equating two operators, namely
the expected value of a quantum random variable and a particular integral with respect to the law of that quantum random variable.

Recall that if $\nu \in \povpm{\H}{X}$, then we call $(X, \borel{X}, \nu)$ a quantum probability space.  If we now consider the quantum random variable $\psi:X \to \B(\H)$, then $\psi^{-1}(B) \in \borel{X}$ for every  $B \in \borel{\B(\H)}$. Thus, the measure $m = m_\psi \in  \povpm{\B(\H)}{\H}$ induced by  $\psi$  
 on $\B(\H)$ and given by
 $$m(B)= \nu(\psi^{-1}(B)), \;\;\; B \in \borel{\B(\H)},$$
 is called the \define{law} of $\psi$.
 
  It is especially important to note that if   $\mu$ is the Borel probability measure on $(X, \borel{X})$ induced by $\nu$ as in~\eqref{povmprobmeas} and
  $\ell$ denotes the Borel probability measure
  on $(\B(\H), \borel{\B(\H)}$
   induced by $m$, then $\ell$ satisfies, for $d$-dimensional Hilbert space,
 $$\ell(B) = \frac{1}{d} \tr(m(B)) = \frac{1}{d} \tr(\nu(\psi^{-1}(B))) = \mu(\psi^{-1}(B)).$$
In this case, since $(X, \borel{X}, \mu)$ and $(\B(\H), \borel{\B(\H)}, \ell)$ are both classical probability spaces, we can immediately conclude from the usual change of variables formula that if $f:\B(\H) \to \B(\H)$ is $m$-integrable in the sense of Definition~\ref{defnofnuintegrable}, then
  \begin{equation}\label{cofveq1}
 \int_X f_\rho(\psi(x)) \d \mu(x)= \int_{\B(\H)} f_\rho(a) \d \ell(a)
 \end{equation}
 where $f_\rho$ is given by~\eqref{muint}.
In fact, with a bit of work we can use this equation to help establish the change of quantum variables formula.

\begin{theorem}[Change of Quantum Variables]\label{cofvthm}
Assume that $\H$ has finite dimension $d$ and that  $(X, \borel{X}, \nu)$ is a quantum probability space. Let $\psi:X \to \B(\H)$ be a 
$\nu$-integrable quantum random variable with law $m \in \povpm{\B(\H)}{\H}$. If $f:\B(\H) \to \B(\H)$ is $m$-integrable,
then
$$\int_{\B(\H)}f(a) \d m(a) = \int_X  f(\psi(x)) \d \nu(x).$$
In particular, 
$\QE{\nu}{\psi} = \mathbf{E}(\psi)$, where
\[
\mathbf{E}(\psi)=\int_{\B(\H)} a \d m(a).
\]
\end{theorem}

 To prove Theorem~\ref{cofvthm}, we require the following lemma.
 
 \begin{lemma}\label{cofvlemma}
If $(X, \borel{X}, \nu)$ is a quantum probability space and $\psi:X \to \B(\H)$ is a quantum random variable with law $m \in \povpm{\B(\H)}{\H}$, then
$$  \frac{\dd \nu}{\dd \mu}(x) = \frac{\dd m}{\dd \ell }(\psi(x))$$
where $\mu$ is the Borel probability measure on $(X, \borel{X})$ induced by $\nu$ as in~\eqref{povmprobmeas} and
  $\ell$ is the Borel probability measure  on $(\B(\H), \borel{\B(\H)}$   induced by $m$.
 \end{lemma}
 
 \begin{proof}
 We know that
 $$\nu(E) = \int_E \frac{\dd \nu}{\dd \mu}(x) \d \mu(x) \;\; \text{for every $E\in\borel{X}$}$$
 and
 $$m(B) = \int_B \frac{\dd m}{\dd \ell }(a) \d \ell(a)\;\;\text{for every $B\in\borel{\B(\H)}$}.$$
 Moreover,  $m(B) = \nu(\psi^{-1}(B))$ so that
\begin{equation}\label{cofvpropeq1}
\nu(\psi^{-1}(B))=  \int_{\psi^{-1}(B)} \frac{\dd \nu}{\dd \mu}(x) \d \mu(x) =  \int_B \frac{\dd m}{\dd \ell }(a) \d \ell(a) = m(B).
\end{equation}
However, using the classical change of variables theorem, we find
\begin{equation}\label{cofvpropeq2}
\int_B \frac{\dd m}{\dd \ell }(a) \d \ell(a) = \int_{\psi^{-1}(B)} \frac{\dd m}{\dd \ell }(\psi(x)) \d \mu(x).
\end{equation}
Thus, we conclude from~\eqref{cofvpropeq1} and~\eqref{cofvpropeq2} that 
$$  \int_{\psi^{-1}(B)} \frac{\dd \nu}{\dd \mu}(x) \d \mu(x)= \int_{\psi^{-1}(B)} \frac{\dd m}{\dd \ell }(\psi(x)) \d \mu(x)$$
for every $B \in \borel{\B(\H)}$. By the uniqueness of the principal Radon-Nikod\'ym derivative, this implies that
$$  \frac{\dd \nu}{\dd \mu}(x) = \frac{\dd m}{\dd \ell }(\psi(x)) $$
as required.
 \end{proof}
 
 We will now complete the proof of the change of variables formula.
 
  \begin{proof}[Proof of Theorem~\ref{cofvthm}]
If $f:\B(\H) \to \B(\H)$, then 
$$f_\rho(a) =\tr \left(
 \rho \left( \frac{\dd m}{\dd \ell }(a)  \right)^{1/2} f(a) \left( \frac{\dd m}{\dd \ell }(a) \right)^{1/2}
\right)$$
by~\eqref{muint}, and so it follows from Lemma~\ref{cofvlemma} that if $a=\psi(x)$, then
\begin{align}\label{cofveq2} \notag
f_\rho(\psi(x)) &=\tr \left(
 \rho \left( \frac{\dd m}{\dd \ell }(\psi(x))  \right)^{1/2} f(\psi(x)) \left( \frac{\dd m}{\dd \ell }(\psi(x)) \right)^{1/2} 
\right)
\\ \notag
&=\tr \left(
 \rho \left(  \frac{\dd \nu}{\dd \mu}(x)  \right)^{1/2} (f\psi)(x) \left(  \frac{\dd \nu}{\dd \mu}(x)\right)^{1/2} 
\right)
\\ 
&= (f\psi)_\rho(x).
\end{align}
Hence, it follows from~\eqref{cofveq1} and~\eqref{cofveq2} that
$$  \int_X (f\psi)_\rho(x) \d \mu(x)
=  \int_X f_\rho(\psi(x)) \d \mu(x) = \int_{\B(\H)} f_\rho(a) \d \ell(a).$$
 Because 
$$\tr\left(\rho \int_X f\psi \d \nu\right) 
=  \int_X (f\psi)_\rho \d \mu
=\int_{\B(\H)} f_\rho \d \ell
=\tr\left(\rho \int_{\B(\H)} f \d m\right) $$
for every state $\rho$, we conclude that
$$ \int_X (f\psi) \d \nu= \int_{\B(\H)} f \d m.$$
In particular, if $f: \B(\H) \to \B(\H)$ is the identity map $f(a)=a$, then
$$\QE{\nu}{\psi} =  \int_X \psi(x) \d \nu(x)= \int_{\B(\H)} a \d m(a)$$
as required.
\end{proof}

\begin{remark}
We can completely mimic the classical notation as follows. 
Let $(X,\borel{X},\nu)$ be a  quantum probability space.
If $A:X \to \B(\H)$ is the quantum random variable $x \mapsto A(x) = a$ and $m = \nu \circ A^{-1}$ is the law of $A$, then 
$\QE{\nu}{A} = \mathbf{E}(A)$ where
$$\mathbf{E}(A)= \int_{\B(\H)} a \d m(a).$$
\end{remark}

\subsection{Chain rules}

We end this section with two results for Radon-Nikod\'ym derivatives that are consequences of the change of  quantum measure theorem.
 
 \begin{theorem}\label{chainrule}
If $\nu_1, \nu_2, \nu_3 \in \povpm{X}{\H}$ with $\nu_1 \ac \nu_2 \ac \nu_3$, then
$$\frac{\dd \nu_1}{\dd\nu_2} \boxtimes \frac{\dd \nu_2}{\dd\nu_3} = \frac{\dd \nu_1}{\dd\nu_3} .$$
\end{theorem}

\begin{proof}
Since $\nu_1 \ac \nu_2$, it follows from Theorem~\ref{NPRNDthm} that
\begin{equation}\label{chaineq1}
\nu_1(E) =  \int_E \frac{\dd \nu_1}{\dd \nu_2} \d \nu_2
\end{equation}
for every $E \in \borel{X}$.
Since $\nu_1 \ac \nu_3$, it follows from Theorem~\ref{NPRNDthm} that
\begin{equation}\label{chaineq2}
\nu_1(E) =  \int_E \frac{\dd \nu_1}{\dd \nu_3} \d \nu_3
\end{equation}
for every $E \in \borel{X}$.
Since $\nu_2 \ac \nu_3$, it follows from Theorem~\ref{quantumchange} that
\begin{equation}\label{chaineq3}
\int_E \frac{\dd \nu_1}{\dd \nu_2} \d \nu_2
=\int_X \left(\frac{\dd \nu_1}{\dd \nu_2} \ch{E} \right) \d \nu_2 
=\int_X \left(\frac{\dd \nu_1}{\dd \nu_2} \ch{E} \right) \boxtimes \frac{\dd \nu_2}{\dd \nu_3}\d \nu_3
=\int_E \frac{\dd \nu_1}{\dd \nu_2}  \boxtimes \frac{\dd \nu_2}{\dd \nu_3}\d \nu_3
\end{equation}
for every $E \in \borel{X}$.
Thus, by equating~\eqref{chaineq1} and~\eqref{chaineq2}, we conclude as a result of~\eqref{chaineq3} that
$$ \int_E \frac{\dd \nu_1}{\dd \nu_2}  \boxtimes \frac{\dd \nu_2}{\dd \nu_3}\d \nu_3 =
\int_E \frac{\dd \nu_1}{\dd \nu_3} \d \nu_3$$
for every $E \in \borel{X}$ which implies that 
$$\frac{\dd \nu_1}{\dd \nu_2}  \boxtimes \frac{\dd \nu_2}{\dd \nu_3} =  \frac{\dd \nu_1}{\dd \nu_3}$$
as required.
\end{proof}

\begin{corollary}
If $\nu_1, \nu_2 \in \povpm{X}{\H}$ with $\nu_2 \ac \nu_1$ and $\nu_1 \ac \nu_2$, then
$$\frac{\dd \nu_1}{\dd\nu_2} \boxtimes \frac{\dd \nu_2}{\dd\nu_1} = \frac{\dd \nu_2}{\dd\nu_1} \boxtimes \frac{\dd \nu_1}{\dd\nu_2} = 1.$$
\end{corollary}

\begin{proof}
If we set $\nu_3 = \nu_1$, then Theorem~\ref{chainrule} implies 
$$\frac{\dd \nu_1}{\dd\nu_2} \boxtimes \frac{\dd \nu_2}{\dd\nu_1} = \frac{\dd \nu_1}{\dd\nu_1} =1.$$
Similarly,
$$\frac{\dd \nu_2}{\dd\nu_1} \boxtimes \frac{\dd \nu_1}{\dd\nu_2} = \frac{\dd \nu_2}{\dd\nu_2}= 1$$
and the proof is complete.
\end{proof}

\section{Quantum conditional expectation and Bayes' rule}\label{condexpsect}

\subsection{Quantum conditional expectation}

We  now introduce the concept of quantum conditional expectation.  Recall that in the classical case
the conditional expectation is defined as a particular random variable possessing certain properties. We show here the quantum analogue has the same feature.

\begin{theorem}\label{condexp}
Suppose that $\H$ has finite dimension, $\nu \in \povpm{X}{\H}$, and that $\psi:X \to \B(\H)_+$ is a $\nu$-integrable quantum random variable with $\QE{\nu}{\psi} \neq 0$.
If
$\F(X)$ is a sub-$\sigma$-algebra of $\borel{X}$, then there exists a function $\varphi:X \to \B(\H)$ such that
\begin{enumerate}
\item $\varphi$ is $\F(X)$-measurable,
\item $\varphi$ is $\nu$-integrable, and 
\item  $\QE{\nu}{\psi \ch{E}} = \QE{\nu}{\varphi \ch{E}}$
for every $E \in \F(X)$.
\end{enumerate}
Moreover, if $\tilde\varphi$ is any other $\nu$-integrable $\F(X)$-measurable function satisfying 
$\QE{\nu}{\psi \ch{E}} = \QE{\nu}{\tilde \varphi \ch{E}}$ for every $E \in \F(X)$, then 
$\nu(\{x \in X : \varphi(x) \neq \tilde\varphi(x)\}) = 0$.
\end{theorem}

\begin{proof}
Suppose that $\nu\in \povpm{X}{\H}$ and that $\psi:X \to \B(\H)_+$ is a $\nu$-integrable quantum random variable.  Let $\nu' = \nu|_{\F(X)}$ be the restriction of $\nu$ to $\F(X)$ so that $\nu' \in \povm{X, \F(X)}{\H}$. Consider the positive operator valued measure  $\tilde\nu \in \povm{X, \F(X)}{\H}$
 defined for $E \in \F(X)$ by
$$\tilde\nu(E) = \int_E \psi \d \nu' = \int_E \psi \d \nu = \QE{\nu}{\psi \ch{E}}.$$
Since  $\tilde\nu \ac \nu'$, it follows from Theorem~\ref{NPRNDthm} that there exists a $\nu'$-integrable $\F(X)$-measurable function
$\varphi$ unique up to sets of $\nu'$-measure 0 such that
$$\tilde\nu(E) = \int_E  \varphi \d \nu'= \int_E \varphi \d \nu = \QE{\nu}{\varphi \ch{E}}$$
for every $E \in \F(X)$. Hence, 
$\tilde\nu(E) =\QE{\nu}{\psi \ch{E}} = \QE{\nu}{\varphi \ch{E}}$
for every $E \in \F(X)$ and the proof is complete. As a final comment, it is perhaps worth noting that  Theorem~\ref{NPRNDthm} also implies that 
 $$\varphi =  \frac{\dd \tilde\nu}{\dd \nu'},$$
the  Radon-Nikod\'ym derivative of $\tilde\nu$ with respect to $\nu'$.
\end{proof}

\begin{defn}
Suppose that $\nu \in \povpm{X}{\H}$ and that  $\psi:X \to \B(\H)_+$ is a quantum random variable with $\QE{\nu}{\psi} \neq 0$.
Suppose further that $\F(X)$ is a sub-$\sigma$-algebra of $\borel{X}$.  A quantum random variable $\varphi:X \to \B(\H)$ satisfying the three properties of Theorem~\ref{condexp} is called a \define{version of quantum conditional expectation of $\psi$ given $\F(X)$ relative to $\nu$} and is denoted by
$\varphi = \QCE{\nu}{\psi}{\F(X)}$.
\end{defn}

\begin{remark}
A consequence of Theorem~\ref{condexp} is that any two versions $\varphi$ and $\tilde\varphi$ of  $\QCE{\nu}{\psi}{\F(X)}$ satisfy $\nu(\{x \in X : \varphi(x) \neq \tilde\varphi(x)\}) = 0$. Thus, instead of saying  ``$\varphi = \QCE{\nu}{\psi}{\F(X)}$ $\nu$-almost surely'' we will identify different versions and say that $\QCE{\nu}{\psi}{\F(X)}$ is \emph{the}  quantum conditional expectation of $\psi$ given $\F(X)$ relative to $\nu$.
Hence, $\varphi=\QCE{\nu}{\psi}{\F(X)}$ is an $\F(X)$-measurable quantum random variable $\varphi:X \to \B(\H)$ with the property that
$$\QE{\nu}{\psi \ch{E}} = \QE{\nu}{\varphi \ch{E}} $$
for every $E \in \F(X)$.
\end{remark}

We now collect several elementary properties of quantum conditional expectation that are notationally analogous to properties of classical conditional expectation.

\begin{proposition} If $\H$ has finite dimension, $\nu \in \povpm{X}{\H}$, and  
$\psi, \psi_1,\psi_2: X \to \B(\H)_+$ are $\nu$-integrable $\F(X)$-measurable quantum random variables such that $\QE{\nu}{\psi} \neq 0$ and
$\QE{\nu}{\psi_j} \neq 0$, for $j=1,2$, then
\begin{enumerate}
\item\label{qce1}
$\QE{\nu}{\psi} = \QE{\nu}{\QCE{\nu}{\psi}{\F(X)}}$,
\item\label{qce2} $\QCE{\nu}{\psi}{\F(X)} = \psi$, and
\item\label{qce3} $\QCE{\nu}{\varrho_1\psi_1 + \varrho_2\psi_2}{\F(X)} = \varrho_1\QCE{\nu}{\psi_1}{\F(X)} + \varrho_2\QCE{\nu}{\psi_2}{\F(X)}$ for all
$\varrho_1, \varrho_2 \in \B(\H)$ that commute with the range of 
$\dd \nu/\dd \mu$.
\end{enumerate}
\end{proposition}

\begin{proof}
The quantum conditional expectation $\varphi= \QCE{\nu}{\psi}{\F(X)}$ satisfies 
$\int_E \psi \d \nu = \int_E \varphi \d \nu$
for every $E \in \F(X)$. Since $X \in \F(X)$, we conclude
$$\QE{\nu}{\psi}  = \int_X \psi \d \nu = \int_X \varphi \d \nu = \int_X \QCE{\nu}{\psi}{\F(X)} \d \nu = \QE{\nu}{\QCE{\nu}{\psi}{\F(X)}}\,, $$
which proves~\eqref{qce1}.

For~\eqref{qce2}, suppose that $\varphi= \QCE{\nu}{\psi}{\F(X)}$. Since $\psi: X \to \B(\H)_+$ is assumed to be $\F(X)$-measurable and $\nu$-integrable, and since it is a tautology that
$\QE{\nu}{\psi \ch{E}} = \QE{\nu}{\psi \ch{E}}$ for every $E \in \F(X)$, we conclude that
$\nu(\{x \in X : \varphi(x) \neq \psi(x)\}) = 0$. Hence, $\QCE{\nu}{\psi}{\F(X)} = \psi$  ($\nu$-almost surely) as required.

Assertion~\eqref{qce3} follows from an application of the definition of quantum conditional expectation and statement~\eqref{qe1} of Theorem~\ref{properties of QE}.
\end{proof}

\subsection{Quantum Bayes' rule}\label{bayessect}
  
We now prove a quantum version of Bayes' rule which can be thought of as a generalisation of the change of quantum measure theorem, or as a special case of the chain rule for  
Radon-Nikod\'ym derivatives. Although it is not immediately apparent from the statement of the quantum Bayes' rule, the random variables which are being multiplied by the $\boxtimes$ operator  are actually Radon-Nikod\'ym derivatives of certain positive operator valued measures (so that Definition~\ref{boxtimes} does, in fact, apply).
 
\begin{theorem}[Quantum Bayes' Rule]\label{bayesthm}
Suppose that $\H$ has finite dimension and that $\nu_1, \nu_2 \in \povpm{X}{\H}$ with   $\nu_2 \ac \nu_1$ and $\nu_1 \ac \nu_2$. 
If $\psi: X \to \B(\H)_+$ is a quantum random variable with $\QE{\nu_2}{\psi} \neq 0$ and $\F(X)$ is a sub-$\sigma$-algebra of $\borel{X}$,  then
\begin{equation}\label{bayesrule}
\QCE{\nu_2}{\psi}{\F(X)}  \boxtimes \QE{\nu_1}{\frac{\dd \nu_2}{\dd \nu_1} \bigg| \F(X)} =
\QE{\nu_1}{\psi \boxtimes \frac{\dd \nu_2}{\dd \nu_1} \bigg|\F(X)}.
\end{equation}
\end{theorem}

\begin{proof}
Suppose that $\F(X)$ is a sub-$\sigma$-algebra of $\borel{X}$. Let $\nu_1' = \nu_1|_{\F(X)}$ be the restriction of $\nu_1$ to $\F(X)$, and similarly let $\nu_2' = \nu_2|_{\F(X)}$. Define the measure $\tilde \nu_1$ by setting
$$\tilde\nu_1(E) = \int_E \psi \boxtimes \frac{\dd \nu_2}{\dd \nu_1} \d \nu_1'$$
for $E \in \F(X)$, and similarly define the measure $\tilde \nu_2$ by setting
$$\tilde\nu_2(E) = \int_E \psi \d \nu_2'$$
for $E \in \F(X)$. Note that $\nu_1', \tilde\nu_1, \nu_2', \tilde\nu_2 \in \povm{X,\F(X)}{\H}$ with $\tilde \nu_1 \ac \nu_1'$ and $\tilde \nu_2 \ac \nu_2'$. 
Using the Radon-Nikod\'ym theorem (Theorem~\ref{NPRNDthm}) combined with the  change of quantum measure theorem (Theorem~\ref{quantumchange}), we conclude that
\begin{equation}\label{bayeseq1}
\tilde\nu_2(E) = \int_E \frac{\dd \tilde\nu_2}{\dd \nu_2'} \d \nu_2' =  \int_E \frac{\dd \tilde\nu_2}{\dd \nu_2'} \boxtimes  \frac{\dd \nu_2'}{\dd \nu_1'} \d \nu_1'
\end{equation}
for every $E \in \F(X)$. Theorem~\ref{NPRNDthm} also implies that
\begin{equation}\label{bayeseq2}
\tilde\nu_1(E) = \int_E \frac{\dd \tilde\nu_1}{\dd \nu_1'} \d \nu_1'
\end{equation}
 for every $E \in \F(X)$.  However, the  change of quantum measure theorem tells us that $\tilde \nu_1 = \tilde \nu_2$; that is,
\begin{equation}\label{bayeseq3}
\tilde\nu_2(E) =  \int_E \psi \d \nu_2' = \int_X \psi \ch{E} \d \nu_2' = 
\int_X (\psi\ch{E})  \boxtimes \frac{\dd \nu_2}{\dd \nu_1} \d \nu_1' =
\int_E \psi \boxtimes \frac{\dd \nu_2}{\dd \nu_1} \d \nu_1' = \tilde \nu_1(E)
\end{equation}
for every $E \in \F(X)$. Thus, as a result of~\eqref{bayeseq3}, by combining~\eqref{bayeseq1} and~\eqref{bayeseq2}, we conclude that
$$ \int_E \frac{\dd \tilde\nu_2}{\dd \nu_2'} \boxtimes  \frac{\dd \nu_2'}{\dd \nu_1'} \d \nu_1' = \int_E \frac{\dd \tilde\nu_1}{\dd \nu_1'} \d \nu_1'$$
for every $E \in \F(X)$ which is to say that
\begin{equation}\label{bayeseq4}
\frac{\dd \tilde\nu_2}{\dd \nu_2'} \boxtimes  \frac{\dd \nu_2'}{\dd \nu_1'} = \frac{\dd \tilde\nu_1}{\dd \nu_1'}.
\end{equation}
However, by Theorem~\ref{condexp} and the definition of quantum conditional  expectation, we know that
$$\QCE{\nu_2}{\psi}{\F(X)} =  \frac{\dd \tilde\nu_2}{\dd \nu_2'}, \;\;\;
 \QE{\nu_1}{\frac{\dd \nu_2}{\dd \nu_1} \bigg| \F(X)} = \frac{\dd \nu_2'}{\dd \nu_1'}, \;\;\; \text{and} \;\;\;
\QE{\nu_1}{\psi \boxtimes \frac{\dd \nu_2}{\dd \nu_1} \bigg|\F(X)} = \frac{\dd \tilde\nu_1}{\dd \nu_1'}$$
which implies that~\eqref{bayeseq4} is equivalent to
$$\QCE{\nu_2}{\psi}{\F(X)}  \boxtimes \QE{\nu_1}{\frac{\dd \nu_2}{\dd \nu_1} \bigg| \F(X)} =
\QE{\nu_1}{\psi \boxtimes \frac{\dd \nu_2}{\dd \nu_1} \bigg|\F(X)}$$
as required.
 \end{proof}
 
\begin{remark}
In classical probability theory,  by taking the trivial $\sigma$-algebra $\F(X)=\{\emptyset, X\}$,  the  change of  measure theorem can be recovered
as a special case of Bayes' rule. However, to establish the same statement in the quantum case requires
$\QE{\nu}{\QE{\nu}{\psi}} = \QE{\nu}{\psi}$, which in general is not true (even for constant
$\psi$). 
\end{remark}

\subsection{Quantum conditional Jensen's inequality}\label{jensensect}
 
 \begin{theorem}[Quantum Conditional Jensen's Inequality] Assume that $\H$ has finite dimension, $\nu\in \povpm{X}{\H}$, and 
$J\subset\mathbb R$ is an open interval containing a closed interval
$[\alpha,\beta]$. If $\psi:X\rightarrow\B(\H)$ is a quantum random variable for which $\psi^*=\psi$
and the eigenvalues of every $\psi(x)$ are contained within $[\alpha,\beta]$, then
$$\QCE{\nu}{\vartheta(\psi)}{\F(X)}  \ge \vartheta\left(\QCE{\nu}{\psi}{\F(X)} \right)$$
for every operator convex function $\vartheta:J\rightarrow\mathbb R$.
\end{theorem}

\begin{proof}
Suppose that $\omega =\QCE{\nu}{\vartheta(\psi)}{\F(X)} $ so that
$$\int_E \omega \d \nu = \int_E \vartheta(\psi) \d \nu$$
for every $E \in \F(X)$.  
Since $\vartheta$ is operator convex, we know from~\cite{FZ} that
\begin{equation}\label{jenseneq1}
\int_E \vartheta(\psi) \d \nu \ge \vartheta \left( \int_E \psi \d \nu \right).
\end{equation}
However, if $\varphi = \QCE{\nu}{\psi}{\F(X)}$ so that 
$$\int_E \varphi \d \nu = \int_E \psi \d \nu$$
for every $E \in \F(X)$, then~\eqref{jenseneq1} implies
$$ \vartheta \left( \int_E \psi \d \nu \right)= 
 \vartheta \left( \int_E \varphi \d \nu \right) $$
   for every $E \in \F(X)$. In other words, we have shown that
$$\int_E \omega \d \nu \ge  \vartheta \left( \int_E \varphi \d \nu \right)$$
  for every $E \in \F(X)$ which implies that
  $$\QCE{\nu}{\vartheta(\psi)}{\F(X)}  \ge \vartheta\left(\QCE{\nu}{\psi}{\F(X)} \right)$$
as required.
\end{proof}

\section{Computing the quantum conditional expectation: an example}\label{explicit}

Assume that $X=\{x_1, x_2, \dots, x_n\}$ so that $\borel{X}$ is the power set of $X$. Consider the sub-$\sigma$-algebra 
$\F(X) = \{\emptyset, \{x_1\}, \{x_2, \ldots, x_n\}, X\}$. 
Choose $\nu \in \povpm{X}{\H}$ with support $X$; thus  $h_j = \nu(\{x_j\})$ is a nonzero positive operator for every
$j=1, \dots, n$. To simplify the discussion, we also assume each $h_j$ is invertible.
Suppose that $\psi: X \to \B(\H)_+$ is a quantum random variable such that $\QE{\nu}{\psi} \neq0$. 
Our aim is  to compute  $\QCE{\nu}{\psi}{\F(X)}$, the quantum conditional expectation of $\psi$ with respect to $\F(X)$ relative to $\nu$. 

Let $\varphi = \QCE{\nu}{\psi}{\F(X)}$ 
so that $\varphi: (X, \F(X)) \to \B(H)$ is an $\F(X)$-measurable quantum random variable with the property that 
$\QE{\nu}{\psi\ch{E}} = \QE{\nu}{\varphi\ch{E}}$
for every $E \in \F(X)$.   Theorem~\ref{condexp} asserts that the quantum random variable $\varphi$ is a particular Radon-Nikod\'ym derivative, which we will now determine explicitly.  
Let $\nu' = \nu|_{\F(X)}$ be the restriction of $\nu$ to $\F(X)$
so that
$\nu'(\{x_1\}) = \nu(\{x_1\}) = h_1$
and
 \begin{equation}\label{ex1oct30eq2}
 \nu'(\{x_2,\ldots, x_n\}) = \nu(\{x_2,\ldots, x_n\}) = \sum_{j=2}^n \nu(\{x_j\}) = \sum_{j=2}^n h_j =1-h_1\,.
 \end{equation}
With
$\mu' = (1/d)\tr \circ \nu' $ we obtain
 \begin{equation}\label{ex1eqA}
\frac{\dd\nu'}{\dd\mu'} = d \left( \frac{\ch{\{x_1\}}} {\tr(h_1)}    \right) h_1 + d \left( \frac{\ch{\{x_2,\ldots, x_n\}}}{\tr(1-h_1)} \right) (1-h_1)\,.
\end{equation}
Define a measure $\tilde\nu:\F(X) \to \B(\H)$ by setting 
$\tilde\nu(E) = \int_E \psi \d \nu'$
for $E \in \F(X)$.  Thus, $\varphi =\QCE{\nu}{\psi}{\F(X)} = \dd \tilde\nu/\dd \nu'$, and our goal
is to compute the value of $\varphi(x_k)$ for $k=1,2,\ldots,n$.
Now with $E = \{x_2, \ldots, x_n\}$, equation~\eqref{ex1eqA} yields
\begin{equation}\label{ex1eqB}
\left( \ch{E}\frac{\dd \nu'}{\dd\mu'}\right)^{1/2} = \ch{E} \sqrt{\frac{d}{\tr(1-h_1) }} (1-h_1)^{1/2}.
\end{equation}
We now claim that there exists an $a \in \B(\H)_+$ such that $\varphi(E) = \{a\}$ which implies that
\begin{equation}\label{claim}
\varphi(x_2) = \cdots = \varphi(x_n) = a.
\end{equation}
To prove the claim, first observe that if a function $\gamma:X \to \R$ has the property that $\gamma(E)$ contains at least three distinct points, then $\gamma$ is not $\F(X)$-measurable. The reason is that if $\alpha_1$, $\alpha_2$, and $\alpha_3$ are distinct values of $\gamma$ on $E$, then $\gamma^{-1}([\alpha_1,\infty))$, $\gamma^{-1}([\alpha_2,\infty))$, and $\gamma^{-1}([\alpha_3,\infty))$ are three distinct subsets of $\F(X)$ different from $\emptyset$ and $X$.  However, $\F(X)$ has only two such subsets.
Suppose now that $a_1, a_2, a_3 \in \B(\H)_+$ are three distinct values of $\varphi$. Thus, because $a_1 \neq a_2$, there is a density operator $\rho \in \state{\H}$ such that $\tr(\rho a_1) \neq \tr(\rho a_2)$. If it is not already true that $\tr(\rho a_3)$ differs from both $\tr(\rho a_1)$ and $\tr(\rho a_2)$, then by the facts that $\state{\H}$ is convex and the functional $a \mapsto \tr(\rho a)$ is continuous, we may perturb $\rho$ slightly in $\state{\H}$ to produce a new density operator, denoted by $\rho$ again, so that 
$\tr(\rho a_1)$, $\tr(\rho a_2)$, and $\tr(\rho a_3)$ are distinct.  But because the map $X \to \R$ defined by $x \mapsto \tr(\rho^{1/2} \psi(x) \rho^{1/2})$ is $\F(X)$-measurable, the previous paragraph tells us that such functions must have at most two values. Hence, $\psi$ is not a quantum random variable on $(X, \F(X))$, which is a contradiction, and thereby proves the claim.
Consequently, for every $\rho \in \state{\H}$, we have
\begin{align*}
 \int_X \ch{E} &\tr \left( \rho \left( \ch{E}\frac{\dd \nu'}{\dd\mu'}\right)^{1/2} \varphi \left( \ch{E}\frac{\dd \nu'}{\dd\mu'}\right)^{1/2} \right) \d \mu'\\
&\qquad\qquad=\sqrt{\frac{d}{\tr(1-h_1) }}\int_X \ch{E} \tr\left(\rho (1-h_1)^{1/2} a (1-h_1)^{1/2}\right) \d \mu'\\
&\qquad\qquad=\tr\left(\rho (1-h_1)^{1/2} a (1-h_1)^{1/2} \right),
\end{align*}
and so we conclude that
\begin{equation}\label{ex1eq6}
\int_{E} \varphi \d \nu' = \int_X \ch{E} \varphi \d \nu' = (1-h_1)^{1/2} a (1-h_1)^{1/2}.
\end{equation}
Because
\begin{equation}\label{ex1eq7}
\int_{E} \varphi \d \nu' = \int_{E} \psi \d \nu' = \sum_{j=2}^n h_j^{1/2} \psi(x_j) h_j^{1/2},
\end{equation}
combining~\eqref{ex1eq6} and~\eqref{ex1eq7}  leads to
$$\QCE{\nu}{\psi}{\F(X)}(x_k) =  \frac{\dd \tilde\nu}{\dd \nu'}(x_k)= \varphi(x_k) =\left(\sum_{j=2}^n h_j \right)^{-1/2}\sum_{j=2}^n h_j^{1/2} \psi(x_j) h_j^{1/2}\left(\sum_{j=2}^n h_j \right)^{-1/2}$$
for $k=2,3,\ldots, n$.

\begin{remark}
Probabilistically, these conditional expectation formul\ae\ do not come as a surprise since  we think of $\QCE{\nu}{\psi}{\F(X)}$ as our ``best guess'' for $\psi$ knowing that $\F(X)$ has happened. Hence,  $\QCE{\nu}{\psi}{\F(X)}(x_1)$, which represents our best guess for $\psi$ knowing that $\{x_1\}$ has happened, is obviously equal to $\psi(x_1)$. On the other hand, if $\{x_2, \ldots, x_n\}$ has happened, then the only information we have is that $\{x_1\}$ has not happened.  Since we have no other information as to which value from $\{x_2, \ldots, x_n\}$ has happened, our best guess for $\psi$ given that  $\{x_2, \ldots, x_n\}$ has happened is the quantum weighted average of the values $\{x_2, \ldots, x_n\}$, namely
$$\left(\sum_{j=2}^n h_j \right)^{-1/2}\sum_{j=2}^n h_j^{1/2} \psi(x_j) h_j^{1/2}\left(\sum_{j=2}^n h_j \right)^{-1/2}.$$
\end{remark}
  
\section*{Acknowledgement}
The work of the authors is supported, in part, by the Natural Sciences and Engineering Research Council of Canada.
The second author thanks the Australian National University for its hospitality during his visit from 
January to May 2011 to the Mathematical Sciences Institute where much of the background material 
for the present paper was learned.


\begin{thebibliography}{00}

\bibitem{farenick2011b}
Douglas Farenick. 
\newblock Arveson's criterion for unitary similarity.
\newblock {\em Linear Algebra Appl.}, 435:769--777, 2011.

\bibitem{FPS}
Douglas Farenick, Sarah Plosker, and Jerrod Smith. 
\newblock Classical and nonclassical randomness in quantum measurements.
\newblock {\em J. Math. Phys.}, 52:122204, 2011.

\bibitem{farenick1992}
D.R.~Farenick. 
\newblock {C$^*$}-convexity and matricial ranges.
\newblock {\em Canad. J. Math.}, 44:280--297, 1992.

\bibitem{FZ}
Douglas R.~Farenick and Fei Zhou. 
\newblock Jensen's inequality relative to matrix-valued measures.
\newblock {\em J. Math. Anal. Appl.}, 327:919--929, 2007.

\bibitem{fuchs--schack2004}
Christopher A.~Fuchs and R\"udiger Schack. 
\newblock Unknown quantum states and operations, a Bayesian view.
\newblock {\em Lecture Notes in Physics}, 649:147--187, 2004.

\bibitem{holevo}
Alexander S. Holevo. 
\newblock {\em Statistical Structure of Quantum Theory}.
\newblock Springer, Berlin, 2001.

\bibitem{kl2004}
Andrei Y.~Khrennikov and Elena R.~Loubenets.
\newblock On relations between probabilities under quantum and classical measurements.
\newblock {\em Found. Phys.}, 34:689--704, 2004. 

\bibitem{kraus}
Karl Kraus.
\newblock {\em States, Effects, and Operations}.
\newblock Springer, Berlin, 1983.

\bibitem{kribs2003}
David W.~Kribs.
\newblock Quantum channels, wavelets, dilations and representations of $\mathcal O_n$.
\newblock {\em Proc. Edinb. Math. Soc.} 46:421--433, 2003.
 
\bibitem{KuboAndo}
Fumio Kubo and Tsuyoshi Ando.
\newblock Means of positive linear operators.
\newblock {\em Math. Ann.}, 246:205--224, 1980.

\bibitem{LS2011a}
M.S.~Leifer and R.W.~Spekkens. 
\newblock Formulating quantum theory as a causally neutral theory of Bayesian inference.
\newblock Preprint, 2011.  Available online at \texttt{arXiv:1107.5849}.

\bibitem{lindblad1999}
G\"oran Lindblad.
\newblock A general no-cloning theorem. 
\newblock {\em Letters Math. Phys.}, 47:189--196, 1999.

\bibitem{nielsen}
Michael A. Nielsen and Isaac L. Chuang.
\newblock {\em Quantum Computation and Quantum Information}.
\newblock Cambridge University Press, Cambridge, 2000.

\bibitem{petz}
D\'enes Petz.
\newblock {\em Quantum Information Theory and Quantum Statistics}.
\newblock Springer, Berlin, 2008.

\bibitem{Pusz}
W. Pusz and S.L. Woronowicz.
\newblock Functional Calculus for sesquilinear forms and the purification map. 
\newblock {\em Rep. Math. Phys.}, 8:159--170, 1975.

\bibitem{SBC}
R\"udiger Schack, Todd A.~Brun, and Carlton M.~Caves.
\newblock Quantum Bayes rule.
\newblock {\em Phys. Rev. A}, 64:014305, 2001.

\bibitem{vedral}
Vlatko Vedral.
\newblock {\em Introduction to Quantum Information Science}.
\newblock Oxford University Press, New York, 2006.


\end{thebibliography}
\end{document}